\documentclass{amsproc}
\usepackage{amssymb}
\usepackage{amscd, tikz-cd}
\usepackage[alphabetic]{amsrefs}
\newtheorem{prop}{Proposition}[section]
\newtheorem{theorem}[prop]{Theorem}
\newtheorem{lemma}[prop]{Lemma}
\newtheorem{corollary}[prop]{Corollary}
\newtheorem{remark}{Remark}
\newtheorem*{thmA}{Theorem A}

\newcommand{\Q}{\mathbb Q}
\newcommand{\Z}{\mathbb Z}
\newcommand{\R}{\mathbb R}
\newcommand{\C}{\mathbb C}
\newcommand{\A}{\mathbb A}
\newcommand{\Sym}{\mathrm{Sym}}
\newcommand{\Lrat}{{\mathcal L}}
\newcommand{\Ree}{\mathrm{Re}}
\newcommand{\sgn}{\mathrm{sgn}}

\DeclareMathOperator{\GL}{GL}

\begin{document}

\title[A comparison of automorphic and Artin L-series of GL(2)-type]{A comparison of automorphic and Artin L-series of GL(2)-type agreeing at degree one primes}

\author{Kimball Martin}
\address[Kimball Martin]{Department of Mathematics, University of Oklahoma, Norman, OK 73019 USA}
\thanks{The first author was partially supported by Simons Collaboration Grant 240605}

\author{Dinakar Ramakrishnan}
\address[Dinakar Ramakrishnan]{Department of Mathematics, Caltech,  Pasadena, CA 91125}
\thanks{The second author was partially supported by NSF grant DMS-1001916}

\dedicatory{To James Cogdell, with friendship and admiration}


\maketitle

\section*{Introduction}

Let $F$ be a number field and $\rho$ an irreducible Galois representation of Artin type, i.e., 
$\rho$ is a continuous $\C$-representation of the absolute Galois group $\Gamma_F$. Suppose $\pi$ is a cuspidal automorphic representation of GL$_n(\A_F)$ such that the $L$-functions
$L(s,\rho)$ and $L(s,\pi)$ agree outside a set $S$ of primes.  (Here, these $L$-functions denote Euler products over just the finite primes, so that we may view them as Dirichlet series in a right half plane.) 
When $S$ is finite, the argument in Theorem 4.6 of \cite{deligne-serre} implies these two $L$-functions in fact agree at all places (cf.\ Appendix A of \cite{martin}).
We investigate what happens when $S$ is infinite of relative degree $\geq 2$, hence density $0$, for the test case $n=2$. Needless to say, if we already knew how to attach a Galois representation $\rho'$ to $\pi$ with $L(s,\rho')=L(s,\pi)$ (up to a finite number of Euler factors), as is the case when $F$ is totally real and $\pi$ is generated by a Hilbert modular form of weight one (\cite{wiles}, \cite{Rogawski-Tunnell}), the desired result would follow immediately from Tchebotarev's theorem, as the Frobenius classes at degree one primes generate the Galois group. Equally, if we knew that $\rho$ is modular attached to a cusp form $\pi'$, whose existence is known for $F=\Q$ and $\rho$ odd by Khare--Wintenberger (\cite{khare}), then one can compare $\pi$ and $\pi'$ using \cite{ramakrishnan-SMO}. However, the situation is more complex if $F$ is not totally real or (even for $F$ totally real) if $\rho$ is even. Hopefully, this points to a potential utility of our approach.

We prove the following

\begin{thmA}
Let $F/k$ be a cyclic extension of number fields of prime degree $p$. Suppose $\rho$ and $\pi$ are as above with their $L$-functions agreeing at all but a finite number of primes $P$ of $F$ of degree one over $k$.
Then $L(s,\rho)=L(s,\pi)$, and 
moreover, at each place $v$, $\pi_v$ is associated to $\rho_v$
according to the local Langlands correspondence.
In particular,
$\pi$ is tempered, i.e., satisfies the Ramanujan conjecture everywhere, and $L(s,\rho)$ is entire.
\end{thmA}

The curious aspect of our proof below is that we first prove the temperedness of $\pi$, and then the entireness of $L(s,\rho)$, before concluding the complete equality of the global $L$-functions. Here are roughly the steps involved. As in the case when $S$ is finite,  consider the quotients $L(s,\pi)/L(s,\rho)$ and $L(s,\pi \times \overline\pi)/L(s,\rho\otimes\overline\rho)$, which reduce in $\Ree(s)>1$ to a quotient of Euler factors involving (outside finitely many) those of degree $\geq 2$. Then we work over $F'= F(\sqrt{-1})$ and, by applying \cite{ramakrishnan}, we find a suitable $p$-power extension $K$ of $F'$ over which even the degree $\geq 2$ Euler factors of the base changes $\pi_K$ and $\rho_K$ agree (outside finitely many), which furnishes the temperedness of $\pi$. Then we carefully analyze certain degree $8$ extensions $K$ of $F$ over which we prove that $L(s,\rho_K)$ is necessarily entire, which is not {\it a priori} obvious. In fact we prove this for sufficiently many twists $\rho\otimes\chi$ for unitary characters $\chi$ of $K$. Applying the converse theorem over $K$, we then conclude the existence of an automorphic form ${}_K\Pi$ of GL$(2)/K$ corresponding to $\rho_K$. We then identify $\pi_K$ with ${}_K\Pi$. The final step is to descend this correspondence $\rho_K \leftrightarrow \pi_K$ down to $F$, first down to $F'$ (by varying $K/F'$) and then down to $F$.

%\medskip

The analytic estimate we use for the inverse roots $\alpha_P$ of Hecke (for $\pi$) is that it is bounded above by a constant times $N(P)^{1/4-\delta}$ for a uniform $\delta >0$, and in fact much stronger results are known for GL$(2)$ (\cite{kim-shahidi}, \cite{kim-sarnak}, \cite{blomer-brumley}). We need the same estimate (of exponent $1/4-\delta$) for GL$(n)$ in general to prove even a weaker analogue of our theorem. (There is no difficulty for $n=3$ if $\pi$ is essentially selfdual.) There is a nice estimate for general $n$ due to Luo, Rudnick and Sarnak (\cite{LRS}), giving the exponent $1-1/2n^2$, but this does not suffice for the problem at hand. We can still deduce that for any place $v$, $L(s,\pi_v)$ has no pole close to $s=0$ to its right, and similarly, we can rule out poles of (the global $L$-function) $L(s,\rho)$ in a thin region in the critical strip.

%\medskip

One can ask the same question more generally for $\ell$-adic representations $\rho_\ell$ of $\Gamma_F$ satisfying the Fontaine-Mazur conditions, namely that the ramification is confined to a finite number of primes and that $\rho_\ell$ is potentially semistable. For the argument of this Note to apply, one would need to know in addition that (i) the Frobenius eigenvalues are pure, i.e., of absolute value $N(P)^{w/2}$ for a fixed weight $w$ for all but finitely many unramified primes $P$, and (ii) for any finite solvable Galois extension $K/F$, the $L$-function of the restrictions of $\rho_\ell$ and $\rho_\ell\otimes\rho_\ell^\vee$ to $\Gamma_K$ admit a meromorphic continuation and functional equation of the expected type, which is known in the Artin case by Brauer. When $F$ is totally real, we will need this for at least any compositum $K$ of a totally real (solvable over $F$) $K'$ with $F(\sqrt{-1})$. For $F$ totally real with $\rho_\ell$ odd ($2$-dimensional), crystalline and sufficiently regular relative to $\ell$, these conditions may follow from the potential modularity results of Taylor.

%\medskip
The first author thanks Jim Cogdell for being a wonderful teacher and for his encouragement.
The second author would like to thank Jim Cogdell for many years of stimulating mathematical conversations and friendship.  We also thank the referee for helpful comments which improved
the exposition.

\section{Notations and preliminaries}

Let $F/k$ be a cyclic extension of number fields of prime degree $p$,
$\rho: \Gamma_F \to \GL_2(\C)$ a continuous, irreducible representation,
and $\pi$ a cuspidal automorphic representation of $\GL_2(\A_F)$ with central
character $\omega_\pi$.

Suppose $L(s, \rho_v) = L(s, \pi_v)$ at almost all places $v$ of $F$ of
 degree 1 over $k$.
Then, at such $v$, $\det(\rho_v) = \omega_{\pi_v}$.  Consequently, $\det(\rho) = \omega_\pi$
globally.  Hence $\omega_\pi$ is finite order, so $\pi$ is unitary.

Now we will recall some basic results which we will use below.  
%First is a lemma of Landau, which roughly
%says that an $L$-function of positive type cannot have zeroes to the right of the first pole.

\begin{lemma}[Landau] \label{landau}
Let $L(s)$ be a Dirichlet series with Euler product which converges in some
right half-plane.  Further suppose that $L(s)$ is of positive type, i.e., that
$\log L(s)$ is a Dirichlet series with non-negative
coefficients.  Let $s_0$ be the infimum of all the real $s_1$ such that $L(s)$ is holomorphic
and non-vanishing in $\Ree(s) > s_1$.  Then, if $s_0$ is finite, $L(s)$ has a pole at $s_0$, and has no zero to the right of $s_0$.
\end{lemma}

In other words, for such an $L(s)$ of positive type, when we approach $s=0$ on the real line from the right, from a real point of absolute convergence, we will not hit a zero of $L(s)$ until we hit a pole.

\medskip

We will also need a suitable (weaker) bound towards the Generalized Ramanujan Conjecture for $\GL(2)$,
which asserts temperedness of $\pi_v$ everywhere, i.e., that $L(s, \pi_v)$ has no poles on
$\Ree(s) > 0$ for all $v$. We will need it for any finite solvable extension $K$ of $F$.

\begin{theorem} \label{GRC-bound}
Let $\pi$ be an isobaric automorphic representation of GL$_2(\A_K)$, for a number field $K$. 
\begin{enumerate}
\item[(a)]If $\pi$ is (unitary) cuspidal, there exists a $\delta < \frac 14$ such that, for any place $v$ of $K$, $L(s, \pi_v)$ has no pole for $\Ree(s) > \delta$.
\item[(b)]If $\pi$ is an isobaric sum of unitary Hecke characters of $K$, then $\pi$ is tempered, and so for any place $v$ of $K$, $L(s, \pi_v)$ has no pole for $\Ree(s) > 0$.
\end{enumerate}
\end{theorem}

The fact that $L(s, \pi_v)$ has no pole for $\Ree(s) \ge \frac 14$ is originally due to Gelbart--Jacquet
\cite{gelbart-jacquet} using $\Sym^2$ $L$-functions.
Subsequently, more precise bounds of $\delta = \frac 19$ were given by Kim--Shahidi
\cite{kim-shahidi},
and later $\delta = \frac 7{64}$ by Kim--Sarnak \cite{kim-sarnak} and Blomer--Brumley
\cite{blomer-brumley}, using respectively the $\Sym^3$ and $\Sym^4$ $L$-functions of $\pi$.

\begin{remark} Some of the estimates towards the Generalized Ramanujan Conjecture
cited above
are just stated for $v$ such that $\pi_v$ unramified.  However, for {\it unitary} cuspidal $\pi$
on $\GL(2)/F$, the general case easily reduces to the unramified situation.
Indeed, if $\pi_v$ is not tempered, we may write it as an irreducible
principal series $\pi_v = \pi(\mu_1 | \det |^t, \mu_2 |\det |^{-t})$, where $\mu_1$ and $\mu_2$
are unitary, and $t$ is real and non-zero.
Since $\check \pi_v \simeq \bar \pi_v$ by unitarity, we see that $\{ \mu_1^{-1} |\det|^{-t}, \mu_2^{-1}
|\det |^t \} = \{ \bar \mu_1 | \det |^t, \bar \mu_2 |\det|^{-t} \}$.  Thus $\mu_1^{-1} = \bar \mu_2 =
\mu_2^{-1}$, i.e., $\mu_1 = \mu_2$.  Hence, $\pi_v = \pi(|\det |^{t}, |\det|^{-t}) \otimes \mu_1$
is just a unitary twist of an unramified principal series. Moreover, as $\mu_1$ is a finite order character times $\vert \cdot \vert^{ix}$ for some $x\in \R$, we may choose a global unitary character $\lambda$ of (the idele classes of) $F$ such that $\lambda_v=\mu_1$, and so $\pi\otimes \lambda^{-1}$ is unitary cuspidal with its $v$-component unramified, resulting in a bound for $t$.
\end{remark}

The next result we need follows from Lemma 4.5 and the proof of Proposition 5.3 in
\cite{ramakrishnan}.

\begin{prop} \label{tcheb-prop}
Fix any $r \ge 1$.  There exists a finite solvable extension $K/k$ containing $F$ such that
any prime $w$ of $K$ lying over a degree $p$ prime $v$ of $F/k$ has degree $\ge p^r$
over $k$.
\end{prop}

The three results above will be used to show $\pi$ is tempered.
In fact, a more explicit version of Proposition \ref{tcheb-prop} for
quadratic fields is used in the other two parts of the proof as well: showing
$L(s,\rho)$ is entire, and showing $\rho$ is modular.
For deducing modularity, we will need the following
version of the converse theorem \`a la Jacquet--Langlands for 
2-dimensional Galois representations.

Denote by $L^*(s, \rho)$ and $L^*(s, \pi)$ the completed $L$-functions for $\rho$ and
$\chi$.

\begin{theorem}\cite[Theorem 12.2]{Jacquet-Langlands} \label{converse}
Suppose $K/F$ is a finite Galois extension and $\rho$ is a 2-dimensional representation of the
Weil group $W_{K/F}$. 
If the $L$-functions $L^*(s, \rho \otimes \chi)$ are entire and bounded
in vertical strips for all idele class characters $\chi$ of $F$, then $\rho$ is modular.
\end{theorem}

In our case of
$\rho$ being an Artin representation, it follows from a theorem of Brauer that each
$L^*(s,\rho \otimes \chi)$ is a ratio of entire functions of finite order.  
Thus knowing $L^*(s,\rho \otimes \chi)$ is entire implies it is of finite order,
whence bounded in vertical strips.  Hence we will only need to check entireness to use
this converse theorem.  

We remark Booker and Krishnamurthy proved a converse theorem
\cite{booker-krishnamurthy} requiring only a weaker hypothesis.

\medskip
We say the global representations $\rho$ and $\pi$ correspond if they do in the sense
of the strong Artin conjecture, i.e., that their local $L$-factors agree almost everywhere.  
For $\GL(2)$, we show that this type of correspondence implies the stronger conclusion that
we assert in Theorem A, that 
$\rho_v$ and $\pi_v$ are associated by the local Langlands correspondence at all $v$.
The local Langlands correspondence for $\GL(2)$ was established by Kutzko \cite{kutzko},
and is characterized uniquely by matching of local $L$- and $\epsilon$- factors of twists
by finite-order characters of $\GL(1)$ (cf.\ \cite[Corollary 2.19]{Jacquet-Langlands}).

\begin{prop} \label{prop:LLC}
Suppose that $L(s, \rho_v) = L(s, \pi_v)$ for almost all $v$.  
Then $\rho_v$ and $\pi_v$ correspond 
in the sense of local Langlands at all $v$ (finite and infinite).
\end{prop}

\begin{proof}
This is a refinement of the argument in Theorem 4.6 of \cite{deligne-serre} and \cite[Appendix A]{martin}.

At a finite place $v$ where $\rho_v$ and $\pi_v$ are both
unramified, this is immediate as $\rho_v$ and $\pi_v$ are determined by $L(s, \rho_v)$
and $L(s, \pi_v)$.  So we only need to show this for $v | \infty$ and $v \in S$, where $S$
is the set of nonarchimedean places $v$ at which $\rho$ or $\pi$ is ramified or 
$L(s, \rho_v) \ne L(s,\pi_v)$.

Observe that, by class field theory, $\det \rho$ corresponds to an idele class character 
$\omega$ over $F$.
From $L(s, \rho_v) = L(s, \pi_v)$, we know $\omega_v = \omega_{\pi_v}$ for all finite $v \not \in S$, and therefore $\omega = \omega_\pi$ by Hecke's strong multiplicity one for $\GL(1)$.  That is to say, 
$\det \rho$ and $\omega_\pi$ correspond via class field theory.

We will use the fact that if $T$ is a finite set of places and $\mu_v$, $v \in T$, are finite-order
characters, there exists a finite-order idele class character $\chi$ globalizing the $\mu_v$'s, i.e., 
$\chi_v = \mu_v$ for $v \in T$.  This is a standard application of the Grunwald--Wang theorem.

First we establish the local Langlands correspondence for $v | \infty$.
Choose a finite-order idele class character $\chi$ which is highly ramified at each $u \in S$
and trivial at each infinite place.  Then, for every $u < \infty$, 
we have $L(s, \rho_u \otimes \chi_u) = L(s, \pi_u \otimes \chi_u)$.  Consequently,
comparing the functional equations for $L(s, \rho \otimes \chi)$ and $L(s, \pi \otimes \chi)$
gives
\begin{equation} \label{eq:fe-ratio}
 \frac{L_\infty(1-s, \bar \pi)}{L_\infty(1-s, \bar \rho)} = 
\frac{\epsilon(s, \pi \otimes \chi)}{\epsilon(s, \rho \otimes \chi)} \frac{L_\infty(s, \pi)}{L_\infty(s, \rho)}.
\end{equation}
Since the poles of the $L$-factors on the right hand side of \eqref{eq:fe-ratio} lie in
$\Ree(s) < \frac 12$, whereas the poles of the $L$-factors on the left hand side lie in
$\Ree(s) > \frac 12$, we can conclude $L_\infty(s, \pi) = L_\infty(s, \rho)$ as they 
must have the same poles (cf.\ \cite[Appendix A]{martin}). 

Note that for $v | \infty$, $\rho_v$ must be a direct sum of two characters with are $1$ or
$\sgn$.  If $v$ is complex, of course $\rho_v = 1 \oplus 1$.  Consequently
\[ L_\infty(s, \pi) = L_\infty(s, \rho) = \Gamma_{\R}(s)^{a} \Gamma_{\R}(s+1)^{b} \]
for some non-negative integers $a, b$.

We claim that, at any archimedean place $v$, the local factors of $\pi$ and
$\rho$ agree. Suppose not. Let $v$ be an errant place, which cannot be complex
since the only option is 
$\Gamma_\C(s)^2 = \Gamma_\R(s)^2\Gamma_\R(s+1)^2$ for either factor. So
$v$ is real.  Since the local factors $L(s,\pi_v)$ are $L(s,\rho_v)$ are different by assumption,
at least one of them, say $L(s, \pi_v)$, 
must be of the form $\Gamma_\R(s)^2$ or $\Gamma_\R(s+1)^2$. 
Call this local factor $G_1(s)$, and the other local factor, say $L(s, \rho_v)$, $G_2(s)$. 
Now twist by a finite-order character $\chi$ of $F$
which is sufficiently ramified at the bad finite places in $S$, equals $1$ at the
archimedean places other than $v$, and at $v$ equals $\sgn$. Then $G_1(s)$ (say
$L(s, \pi_v \otimes \sgn)$) becomes
$G_1(s+\delta)$, with $\delta\in\{1, -1\}$.  
Similarly we see that $G_2(s)$ (say $L(s, \rho_v \otimes \sgn)$)  becomes 
$G_2(s+\delta')$ with $\delta' \in \{ 1, 0, -1 \}$, where $\delta' = 0$ if and only if 
$G_2(s)=\Gamma_\R(s) \Gamma_\R(s+1)$.
The total archimedean identity 
$L_\infty(s,\pi \otimes \chi) = L_\infty(s,\rho \otimes \chi)$
persists in this case, and comparing with $L_\infty(s, \pi) = L_\infty(s, \rho)$ gives
$G_1(s+\delta)/G_1(s) = G_2(s+\delta')/G_2(s)$. 
The only way this can happen is if $G_1(s+\delta)=G_2(s+\delta')$, which forces
$G_1(s)=G_2(s)$, contradicting the
assumption. Hence the claim.

This in fact implies that $\pi_v$ is associated to $\rho_v$ by the local Langlands correspondence
for $v$ archimedean.  Namely, for complex $v$, $L(s,\pi_v) = \Gamma_\C(s)$ and $\pi_v$
unitary (or $\omega_{\pi_v} = \det \rho_v = 1$) implies
$\pi_v = \pi(1,1)$.  For real $v$, if $\pi_v = \pi(\mu_1, \mu_2)$ is a unitarizable principal series
then $\mu_i = \sgn^{m_i} | \cdot |^{s_i}$ with $|s_i| < \frac 12$.  Hence $L(s, \mu_i) = \Gamma_\R(s)$ implies $\mu_i = 1$ and $L(s, \mu_i) = \Gamma_\R(s+1)$ implies $\mu_i = \sgn$.  Consequently, if 
$L(s, \pi_v) = \Gamma_\R(s)^c \Gamma_\R(s+1)^d$ with $c+d=2$, then $\pi_v$ is an isobaric
sum of $c$ copies of $1$ and $d$ copies of $\sgn$, and therefore matches $\rho_v$ in the sense
of Langlands.

Finally, consider a finite place $v \in S$.  Let $\mu$ be a finite-order character of $F_v^\times$.  Let $\chi$
be an idele class character which is highly ramified at all $u \in S - \{ v \}$ such that 
$\chi_v = \mu$.  For all $u \not \in S$, we have 
$L(s, \pi_u \otimes \chi_u) = L(s, \rho_u \otimes \chi_u)$ 
and $\epsilon(s, \pi_u \otimes \chi_u, \psi_u) = \epsilon(s, \rho_u \otimes \chi_u, \psi_u)$
by the local Langlands correspondence.  But the same is also true for $u \in S - \{ v \}$ by a result
of Jacquet and Shalika \cite{JS} which is often called stability of $\gamma$-factors:
for twists by sufficiently ramified characters, the $L$-factors are 1 and 
the $\epsilon$-factors are equal since $\omega_{\pi_u} = \det_{\rho_u}$.
Hence the same comparison of functional equations that led to \eqref{eq:fe-ratio} in the 
archimedean case gives us
\begin{equation}
\frac{L(1-s, \bar \pi_v \otimes \bar \mu)}{L(1-s, \bar \rho_v \otimes \bar \mu)}
= \frac{\epsilon(s, \pi_v \otimes \mu, \psi_v)}{\epsilon(s, \rho_v \otimes \mu, \psi_v)}
\frac{L(s, \pi_v \otimes \mu)}{L(s, \rho_v \otimes \mu)}.
\end{equation}
Again, a comparison of the poles implies $L(s, \pi_v \otimes \mu) = L(s, \rho_v \otimes \mu)$,
and similarly that $L(1-s, \bar \pi_v \otimes \bar \mu) = L(1-s, \bar \rho_v \otimes \bar \mu)$.
Consequently, $\epsilon(s, \pi_v \otimes \mu, \psi_v) = \epsilon(s, \rho_v \otimes \mu, \psi_v)$.
Since this is true for all $\mu$, we conclude that $\pi_v$ and $\rho_v$ must correspond in the
sense of local Langlands. 
\end{proof}

\section{Temperedness}

To show $\pi$ is tempered, we will make use of solvable base change.

Let $K$ be a solvable extension of $F$.  Denote by $\rho_K$ the restriction of $\rho$ to
$\Gamma_K$.  Denote by $\pi_K$ the base change of $\pi$ to $\GL_2(\A_K)$,
whose automorphy we know by Langlands (\cite{langlands}). More precisely, $\pi_K$ is either cuspidal or else an isobaric sum of two unitary Hecke characters of $K$.

Put
\begin{equation}
\Lambda_K(s) = \frac{L^*(s, \pi_K \times \bar \pi_K)}{L^*(s, \rho_K \times \bar \rho_K)}.
\end{equation}
If there is no confusion, we will write $\Lambda$ instead of $\Lambda_K$.
%\marginpar{Edited paren rem for clarity}
This has a factorization,
\[ \Lambda(s) = \prod_{v} \Lambda_{v}(s),
\quad \text{where } \Lambda_{v}(s) =  \frac{L(s, \pi_{K,v} \times \bar \pi_{K,v})}
{L(s, \rho_{K,v} \times \bar \rho_{K,v})}, \]
for any place $v$ of $K$, and is {\em a priori}  analytic with no zeroes for $\Ree(s) > 1$.
What is crucial for us is that the logarithms of the numerator and denominator of
the non-archimedean part of $\Lambda$
are Dirichlet series with positive coefficients, so we will be able to apply
Landau's lemma.

For an arbitrary set $S$ of places of $K$, we write
 $\Lambda_{S}(s) = \prod_{v \in S} \Lambda_{v}(s)$.
Denote by $S_j$ the set of finite places $v$ of $K$ of degree $j$ over $k$ for which $\rho_K$,
$\pi_K$ and $K$ are unramified, but $\rho_{K,v}$ and $\pi_{K,v}$ do not correspond.
%\marginpar{Added $K$ to subs for $\pi$ and $\rho$}
Then we can write
\begin{equation} \label{eq:Lambda-prod}
 \Lambda(s) = \Lambda_{\Sigma}(s) \prod_{j \ge 2} \Lambda_{S_j}(s),
\end{equation}
where $\Sigma$ is a finite set containing $S_1$, the archimedean places, and
the set of finite places where $\pi$, $\rho$ or $F$ is ramified.

Then $\Lambda(s)$ satisfies a functional equation of the form
\begin{equation} \label{eq:Lambda-FE}
\Lambda(s) = \epsilon(s)   \Lambda(1-s),
\end{equation}
where $\epsilon(s)$ is an invertible holomorphic function on $\C$.

To show $\pi$ is tempered, we will use the following lemma in two different places.

\begin{lemma} \label{tempered-lemma}
Let $\delta$ be as in Theorem \ref{GRC-bound}, and $S_j$ as above for $K=F$.
%\marginpar{added explaination of $S_j$}
Then

(i)  $L_{S_j}(s, \pi \times \bar \pi)$ has no poles or zeroes on $\Ree(s) > \frac 1j + 2 \delta$; and

(ii) $L_{S_j}(s, \rho \times \bar \rho)$ has no poles or zeroes on $\Ree(s) > \frac 1j$.
\end{lemma}

\begin{proof}
Let us prove (i).

Suppose $v \in S_j$.  Let $\alpha_{1, v}$ and $\alpha_{2, v}$ be the Satake parameters for
$\pi_v$ and $\beta_v = \max \{ | \alpha_{1, v} |, | \alpha_{2, v} | \}$.  Then
\begin{equation} \label{eq:Lpiv-bound}
 \log L(s, \pi_v \times \bar \pi_v) \le \sum_{i=1}^2 \sum_{l=1}^2 \sum_{m \ge 1}
\frac{(\alpha_{i,v} \bar \alpha_{l, v})^m}{mq_v^{ms}}
\le 4 \sum_{m \ge 1} \left( \frac{\beta_v^{2}}{q_v^{s}} \right)^m \le 4 \sum_{m \ge 1}
 \frac 1{q_v^{(s-2\delta)m}},
\end{equation}
where the last step follows from the bound in Theorem \ref{GRC-bound}, which
is equivalent to $\beta_v < q_v^\delta$.
In particular, the local factor $L(s, \pi_v \times \bar \pi_v)$ (which is never zero) is holomorphic for $\Ree(s) > 2 \delta$.

Let $p_v^{f_v}$ denote the norm of the prime of $k$ below $v$, where $p_v$ is a rational
prime.  Then from \eqref{eq:Lpiv-bound},
we see
\begin{align*}
 \log L_{S_j}(s, \pi \times \bar \pi) & \le 4 \sum_{v \in S_j} \sum_{m \ge 1} \frac 1{p_v^{f_v j(s-2/9)m}}
  \le 4 \sum_{v \in S_j} \sum_{m \ge 1} \frac 1{p_v^{j(s-2/9)m}}  \\
 & \le 4[k:\Q] \sum_{p_i} \sum_{m \ge 1} \frac 1{(p_i^m)^{j(s-2/9)}} \le 4 \sum_{n \ge 1} \frac 1{n^{j(s-2/9)}},
\end{align*}
where $p_i$ runs over all primes in the penultimate inequality.
This series converges absolutely, and uniformly in compact subsets of the region of $s\in \C$ with
\begin{equation}
\Ree(s) > \frac 1j + 2\delta.
\end{equation}
Since $L_{S_j}(s, \pi \times \bar \pi)$ is of positive type, we may apply Landau's lemma (Lemma \ref{landau}) to conclude that it is also {\it non-zero} and holomorphic for $\Ree(s) > \frac 1j + 2\delta$,
which implies (i). (Since this incomplete $L$-function has infinitely many Euler factors, it is not obvious that it is non-zero in this region without applying Landau.)

The argument for (ii) is the same, except that one uses the fact that the Frobenius
eigenvalues for $\rho_v$ lie on $\Ree(s) = 0$ in place of Theorem \ref{GRC-bound}.
\end{proof}

\begin{prop} \label{tempered-prop}
For each place $v$ of $F$, $\pi_v$ is tempered.
\end{prop}

\begin{proof}
By Theorem \ref{GRC-bound}, for any finite set $\Sigma$,  $\Lambda_\Sigma$ has no poles for
$\Ree(s) \ge \frac 12$.

Fix $N$ such that $N > (\frac 12 - 2\delta)^{-1}$.
Then Lemma \ref{tempered-lemma}
implies that $\Lambda_{S_j}(s)$ has no poles (or zeroes) in the region
$\Ree(s) \ge \frac 12$ for $j \ge N$.
Fix $r$ such that $p^r \ge N$.  By Proposition \ref{tcheb-prop},
there is a solvable $K/F$ such that each prime $v$ of degree $p$ in $F/k$
splits into primes all of degree $\ge p^r$ in $K/k$.
Then for all but finitely many primes $v$ of degree $< N$ in $K/k$, we have $\Lambda_{K,v}(s) = 1$,
so in fact we can rewrite \eqref{eq:Lambda-prod} as
\begin{equation} \label{eq:Lambda-prod2}
 \Lambda_K(s) = \Lambda_{K,\Sigma}(s) \prod_{j \ge N} \Lambda_{K,S_j}(s),
\end{equation}
for a finite set $\Sigma$.

Hence Lemma \ref{tempered-lemma} implies
$\Lambda_K$ has no zeroes or poles for $\Ree(s) \ge \frac 12$.
Therefore, by the functional equation \eqref{eq:Lambda-FE},
$\Lambda_K$ is entire and nowhere zero.

Now suppose $\pi_v$ is non-tempered for some place $v$ of $F$.  Then $\pi_v$ is an
irreducible principal series $\pi_v = \pi(\mu_1, \mu_2)$ and $L(s, \pi_v) = L(s, \mu_1)L(s,\mu_2)$,
so $L(s, \pi_v \times \bar \pi_v) = \prod_{i=1}^2 \prod_{j=1}^2 L(s, \mu_i \bar \mu_j)$.

Consider first $v < \infty$.  Write $L(s, \mu_i) = (1-\alpha_{i,v} q_v^{-s})^{-1}$ for $i = 1,2$.
%\marginpar{changed $\chi_i$ to $\mu_i$}
Interchanging $\alpha_{1,2}$ and $\alpha_{2,v}$ if necessary,
we may assume $|\alpha_{1,v}| > 1$.
Then for any place $w$ of $K$ above $v$,
$L(s,\pi_{K,w})^{-1}$ has a factor of the form $(1-\alpha_{1,v}^{f_w} q_w^{-s})$
where $f_w = [K_w:F_v]$.  Hence $L(s, \pi_{K,w} \times \bar \pi_{K,w})^{-1}$ has
$(1-|\alpha_{1,v}|^{2f_w} q_w^{-s})$ as a factor.
Looking at the Dirichlet series $\log L(s, \pi_{K,w} \times \bar \pi_{K,w})$, Landau's lemma
tells us  $L(s, \pi_{K,w} \times \bar \pi_{K,w})$ has a pole at $s_0 =
\frac {2\log |\alpha_{1,v}|}{\log q_v}$.

Now consider $v | \infty$.  We may assume $L(s, \mu_1)$ has a pole to the right of $\Ree(s) = 0$.
Again suppose $w$ is a place of $K$ with $w | v$.  Either
$L(s, \pi_{K,w} \times \bar \pi_{K, w})$ is $L(s, \mu_1 \bar \mu_1)$ (when $v$ splits in $K$) or
it equals $L(s, \mu_{1, \C} \bar \mu_{1, \C})$, where $\mu_{1, \C}(z) = \mu_1(z \bar z)$.  If we choose
$s_0 > 0$ so that both $L(s, \mu_1 \bar \mu_1)$ and $L(s, \mu_{1,\C} \bar \mu_{1,\C})$ have
poles in $\Ree(s) \ge s_0$, then the same is true for $L(s, \pi_{K,w} \times \bar \pi_{K, w})$.

Hence, in either case $v < \infty$ or $v | \infty$, there exists $s_0 > 0$, independent of $K$,
such that $L(s, \pi_{K,w} \times \bar \pi_{K,w})$ has a pole in $\Ree(s) \ge s_0$.
Another application of Landau's lemma tells us
$L_\Sigma(s, \pi_K \times \bar \pi_K) \prod_{j \ge N} L_{S_j}(s, \pi_K \times \bar \pi_K)$
must also have a pole in $\Ree(s) \ge s_0$.

Since $\Lambda_K$ is entire, it must be that the denominator
$L_\Sigma(s, \rho_K \times \bar \rho_K) \prod_{j \ge N} L_{S_j}(s, \rho_K \times \bar \rho_K)$
of $\Lambda_K(s)$ also has a pole at $s_0$.
%\marginpar{changed $s_2$ to $s_0$}
However, $L_\Sigma(s, \rho_K \times \bar \rho_K)$
has no poles
to the right of $\Ree(s) = 0$, and by Lemma \ref{tempered-lemma},
$L_{S_j}(s, \rho_K \times \bar \rho_K)$
has no poles to the right of $\Ree(s) = \frac 1j$. Hence $s_0 \le \frac 1j$.
%\marginpar{changed $<$ to $\le$}

Take $r_0 \ge r$ such that $p^{r_0} > \frac 1{s_0}$.   By Proposition \ref{tcheb-prop},
we may replace $K$ by a larger
solvable extension so that every prime of degree $p$ in $F$ splits into primes of
degree at least $p^{r_0}$ in $K$.  But then the denominator of $\Lambda_K(s)$ has no poles
to the right of $p^{-r_0} < s_0$, a contradiction.
\end{proof}

\section{Entireness}

Now we would like to deduce that $L(s, \rho)$ is entire.  However, we cannot
directly do this when $F/k$ is quadratic, but only over a
quadratic or biquadratic extension $K$.  Consequently we deduce that
$\rho$ and $\pi$ correspond over $K$.  In the final section, we will deduce that
$\rho$ and $\pi$ correspond over $F$, which will imply the entireness of $L(s,\rho)$.
From here on, $i$ will denote a primitive fourth root of $1$ in $\overline F$.

\begin{prop} \label{prop2}
There exists a solvable extension $K/k$ containing $F$, depending
only on $F/k$, such
that $L^*(s, \rho_K)$ is entire.  In fact, if $p=[F:k] > 2$, we can take $K=F$.  If
$p = 2$ and $i \in F$, then we can take $K/k$ to be cyclic of degree 4.
If $p=2$ and $i \not \in F$, then we can take $K/k$ so that $K/F$ is biquadratic.
\end{prop}

We will actually construct the field $K$ in the proof, and the explicit construction will be used
in the next section to prove our main theorem.

\begin{proof}
Consider the ratio of $L$-functions
\begin{equation}
 \Lrat_K(s) = \frac{L^*(s, \rho_K)}{L^*(s, \pi_K)},
\end{equation}
which is analytic for $\Ree(s) > 1$.
Since $L(s, \pi_K)$ has no poles, to show $L^*(s, \rho_K)$ is entire,
it suffices to show $\Lrat_K(s)$ has no poles.

As with $\Lambda$, we may write $\Lrat_K = \prod_v \Lrat_{K,v}$
and define $\Lrat_{K,S} = \prod_{v \in S} \Lrat_{K,v}$.   Then note that
we can write
\begin{equation} \label{Lrat-fact}
\Lrat_K(s) = \Lrat_{K, \Sigma}(s) \prod_{j > 1} \frac{L_{S_j}(s,\rho_K)}{L_{S_j}(s, \pi_K)},
\end{equation}
for some finite set $\Sigma$.  As with $\Lambda$, it satisfies a functional equation of the
form
\begin{equation}
\Lrat_K(s) = \epsilon_{\Lrat,K}(s) \Lrat_K(1-s),
\end{equation}
with $\epsilon_{\Lrat,K}(s)$ everywhere invertible.
Hence it suffices to show $\Lrat_K(s)$ is analytic in $\Ree(s) \ge \frac 12$.

Note $\Lrat_{K, \Sigma}(s)$ is analytic for
$\Ree(s) \ge \frac 12$ (in fact for $\Ree(s) > 0$, since $\pi$ must be tempered).
The same argument as in Lemma \ref{tempered-lemma}
shows that $L_{S_j}(s, \rho_K)$ and $L_{S_j}(s, \pi_K)$ are both holomorphic and
never zero in $\Ree(s) > \frac 1j$, using the fact that now we know $\pi$ is tempered
(Proposition \ref{tempered-prop}).   Hence if $p > 2$, then, already for $K=F$, \eqref{Lrat-fact}
implies that $\Lrat_K(s)$ is analytic in $\Ree(s) > \frac 1p$, and we are done.
%\marginpar{changed $\frac 13$ to $\frac 1p$}

Now suppose $p=2$.  Here we use a more explicit version of Proposition
\ref{tcheb-prop} for cyclic $p^2$-extensions:  if $K \supset F \supset k$ is a chain of cyclic
$p$-extensions with $K/k$ cyclic, then every unramified inert prime $v$ in $F/k$ lies under
a (unique) degree $p^2$ prime $w$ in $K/k$ \cite[Lemma 4.4]{ramakrishnan}.

Write $F = k(\sqrt \alpha)$ with $\alpha \in k$.   If $i \in F$, then $K = k(\alpha^{1/4})$
is a cyclic extension of degree 4.  Hence the lemma just quoted means that
$\Lrat_K(s)$ is analytic in $\Ree(s) > \frac 14$ and we are done.

We may therefore assume $i \not \in F$.  Put $F_1 = k(i)$ and $F_2 = k(\sqrt{-\alpha})$.
Let $E = k(i, \sqrt \alpha) = F(i)$ be the compositum of these fields, which is biquadratic over $k$.
For $i=1, 2$, let $\alpha_i \in F_i$ such that $E = F_i(\sqrt{\alpha_i})$, and
put $K_i = F_i(\alpha_i^{1/4})$.  Then $K_i/F_i$ is cyclic of degree 4 with
with $E$ as the intermediate subfield ($i = 1, 2$).  Denote by $K$ the compositum $K_1 K_2$.
(This construction of $K$ is the $p=2$ case of a construction given in
 \cite[Section 5]{ramakrishnan}.)  Here is a diagram for the case $i \not \in F$.

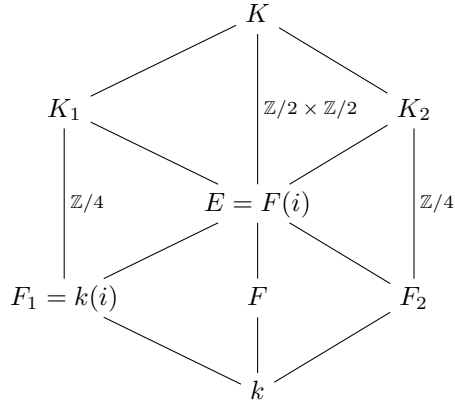
\begin{center}
\begin{tikzcd}[every arrow/.append style={-}]
{} & K  \arrow{ld} \arrow{rd}  \arrow{dd}{\Z/2 \times \Z/2}  & {}  \\
K_1 \arrow{dr}  \arrow{dd}{\Z/4} &
{} & K_2 \arrow{ld}  \arrow{dd}{\Z/4}  \\
{} & E = F(i) \arrow{d} \arrow{ld} \arrow{rd}  & {} \\
F_1 = k(i) \arrow{dr} &  F \arrow{d} & F_2 \arrow{ld} \\
{} & k & {}
\end{tikzcd}
\end{center}

Fix any prime $v$ of degree 2 in $F/k$.  We claim any prime $w$ of $K$
above $v$ has degree $\ge 4$ over $k$.  If not, $v$ splits into 2 primes
$v_1, v_2$ in $E$.  Say $w | v_1$.   By
\cite[Lemma 5.8]{ramakrishnan}, $v_1$ is degree 2 either over $F_1$ or over $F_2$.
Let $i \in \{ 1, 2 \}$ be such that $v_1$ is degree 2 over $F_i$, and let $u$ be the prime of
$F_i$ under $v_i$.  Then $u$ is inert in $K_i/F_i$, and so $v_1$ is inert in $K_i/E$.  Hence
$w$ has degree $\ge 2$ in $K/E$, and therefore degree $\ge 4$ in $K/k$, as claimed.

Thus, for such $K$, $\Lrat_K(s)$ is analytic in $\Ree(s) > \frac 14$.
\end{proof}

\begin{remark} The above argument for $p > 2$ in fact shows the following: if $F/k$ is any
extension, $\pi$ is tempered, and $\rho$ corresponds to $\pi$ at all but finitely many places of degree
$\le 2$, then $L^*(s, \rho)$ is entire.
\end{remark}

\begin{corollary}
With $K$ as in the previous proposition, and $\chi$ and idele class character of $K$.
Then $L^*(s, \rho_K \otimes \chi)$ is entire.
\end{corollary}

\begin{proof}
First suppose $p > 2$, so we can take $K=F$.  Then if $\chi$ is an idele class character of $F$,
$\rho \otimes \chi$ and $\pi \otimes \chi$ locally correspond at almost all degree 1 places,
so the previous proposition again applies to  say $L^*(s, \rho \otimes \chi)$ is entire.

Now suppose $p = 2$, and let $K$ be an extension as constructed in the proof of the above
proposition.  Then, for all but finitely many places of degree $\le 2$, we have that $\rho_K$ and $\pi_K$
locally correspond.  Hence the same is true for $\rho_K \otimes \chi$ and $\pi_K \otimes \chi$
for an idele class character $\chi$ of $K$.  The above remark  then gives the corollary.
\end{proof}

\begin{corollary} With $K$ as in proposition above,
the Artin representation $\rho_K$ is automorphic, and corresponds
to the base change $\pi_K$.
\end{corollary}

\begin{proof}
The automorphy of $\rho_K$ now follows from the previous corollary and the converse
theorem (Theorem \ref{converse}).   Furthermore, since $\rho_K$ corresponds to an automorphic
representation which agrees with $\pi_K$ at a set of places of density 1, it must agree with
$\pi_K$ everywhere by \cite{ramakrishnan-SMO}.
\end{proof}

\section{Descent}

In this section, we will finish the proof of Theorem A, that $\rho$ corresponds to $\pi$
globally.  The last corollary (together with Proposition \ref{prop:LLC})
already tells us this is the case if $p > 2$, so we may assume
$p=2$ in what follows.  In the previous section, we showed $\rho_K$ corresponds to $\pi_K$
over a suitable
quadratic (when $i \in F$) or biquadratic (when $i \not \in F$) extension $K/F$.
Here we will show that by varying our choice of $K$, we can descend this correspondence
to $F$.

As above, let $S_2$ be the set of finite unramified primes $v$ of degree 2 in $F/k$ at which
$\rho_{v}$ and $\pi_v$ are both unramified but do not correspond.
If $S_2$ is finite, then an argument of Deligne--Serre \cite{deligne-serre}
(cf.\ \cite[Appendix A]{martin} or Proposition \ref{prop:LLC})
already tells us $\rho$ and $\pi$ correspond everywhere (over $F$ at finite places and
at the archimedean place over $\Q$), so
we may assume $S_2$ is infinite.
As before, write $F$ as $k(\sqrt{\alpha})$ for some $\alpha \in k$.

\medskip
First suppose that $F$ contains $i$.  Then, as in the proof of Proposition \ref{prop2},
 $K = k(\alpha^{1/4})$ is a quadratic extension of $F$, with $K$ cyclic over $k$,
 and $\pi_K$ corresponds to $\rho_K$ globally.
 Let $T_2$ denote the subset of $S_2$ consisting of finite places $v$ of $F$
which are unramified in $K$. Clearly, the
complement of $T_2$ in $S_2$ is finite.

We claim that $T_2$ is empty, so that $\rho$ and $\pi$
correspond {\it everywhere}, already over $F$. Suppose not. Pick any element $v_0$ of $T_2$.
Since det$(\rho)$ corresponds to $\omega_\pi$ at all places, and since $\rho_K$ and $\pi_K$ correspond
exactly, if $\chi_{v_0}$ is the quadratic character of $F_{v_0}$ attached to the
quadratic extension $K_{\bar v_0}$, $\rho_{v_0}$ must correspond to $\pi_{v_0}\otimes
\chi_{v_0}$.
As $K_{\bar v_0}/F_{v_0}$ is
unramified, so is $\chi_{v_0}$.

Now we may modify the choice of $K$ as
follows. Pick a $\beta$ in $k$ which is a square but not a fourth power, and put $\tilde
K= k((\alpha\beta)^{1/4})$. Then $\tilde K$ contains $F$ and is a cyclic quartic
extension of $k$ (just like $K$), so that all but a finite number of places in $S_2$ are inert in $\tilde K$.
But now we may choose $\beta$ such that $v_0$ ramifies in
$\tilde K$. It follows (as above) that over $\tilde K$, the base changes of of $\pi$ and $\rho$
correspond everywhere, and that $\rho_{v_0}$ corresponds to $\pi_{v_0}\otimes
\tilde\chi_{v_0}$, where $\tilde\chi_{v_0}$ is now the ramified local character at $v_0$ (attached
to $\tilde{K}/ F$). This gives a contradiction as  we would then need
$\pi_{v_0}\otimes\chi_{v_0}$ to be isomorphic to $\pi_{v_0}\otimes \tilde \chi_{v_0}$ (only the
latter twist is ramified). So the only way to resolve this
is to have $\rho_{v_0}$ correspond to $\pi_{v_0}$. Then $v_0$ cannot lie in $T_2$. As it was
taken to be a general element of $T_2$, the whole set $T_2$ must be empty, proving the
claim. We are now done if $i$ belongs to $F$.

\medskip
Thus we may assume from here on that $i \not \in F$. Then $E=F(i)$ is a
biquadratic extension of $k$, and every element $v$ of $S_2$ which is unramified in $E$
must split there, say into $v_1$ and $v_2$.  So it suffices to prove that $\rho$ and $\pi$
correspond exactly over $E$.  Suppose not.  Without loss of generality,
say $\rho_{E,v_1}$ and $\pi_{E,v_1}$ do not correspond.

Consider the construction of $K$, biquadratic over $F$, with subfields $K_1, K_2$,
$E$, $F_1$, $F_2$ as in the proof of Proposition \ref{prop2}.  Interchanging indices
on $K_i$'s and $F_i$'s if necessary, we may assume  that
$v_1$ has degree 2 in $E/F_1$, so it lies under a unique prime $w$ in
$K_1 = F_1(\alpha_1^{1/4})$.  Since $w$ splits in $K$, $\rho_{K_1, w}$ must correspond to
$\pi_{K_1, w} \otimes \chi$, where $\chi$ is the unramified quadratic character attached to
$K_{1,w}/E_{v_1}$.

As in the previous case, we may modify $K_1$ to $\tilde K_1 = F_1((\alpha_1 \beta_1)^{1/4})$
for some $\beta_1 \in F_1$ which is a square but not a fourth power.   We may choose
$\beta_1$ so that $\tilde K_1/E$ is ramified.
Then $\tilde K_1/F_1$ is a cyclic extension of degree 4 containing $E$, and Proposition
\ref{prop2} is still valid with $K$ replaced by $\tilde K = \tilde K_1 K_2$.  Let $\tilde w$
be the prime of $\tilde K_1$ over $v_1$.  Since $\tilde w$ splits over $\tilde K$, we
also have that $\rho_{\tilde K_1, w}$ corresponds to $\pi_{\tilde K_1, w} \otimes \tilde \chi$,
where $\tilde \chi$ is the ramified quadratic character associated to $\tilde K_{1,w}/E_{v_1}$.
Again, this implies $\pi_{K_1, w} \otimes \chi \simeq \pi_{\tilde K_1, w} \otimes \tilde \chi$,
which is a contradiction.
\qed

\end{document}